\documentclass[12pt, leqno]{article}

\usepackage{amsmath, amssymb, amscd, verbatim, xspace,amsthm}
\usepackage{latexsym, epsfig, color}

\newcommand{\A}{\ensuremath{{\mathbb{A}}}}

\newcommand{\Z}{\ensuremath{{\mathbb{Z}}}\xspace}
\renewcommand{\P}{\ensuremath{{\mathbb{P}}}}

\newcommand{\ra}{\rightarrow}
\newcommand{\lra}{\longrightarrow}
\newcommand\Det{\operatorname{Det}}
\newcommand\Hom{\operatorname{Hom}}

\newcommand\im{\operatorname{im}}

\newcommand\Sym{\operatorname{Sym}}

\newcommand\tensor{\otimes}
\newcommand\isom{\cong}
\newcommand\sub{\subset}
\newcommand\tesnor{\otimes}
\newcommand\wn{\wedge^n}

\newcommand\GL{\operatorname{GL}}

\newcommand\Spec{\operatorname{Spec}}
\newcommand\Proj{\operatorname{Proj}}

\newcommand\wt{\wedge^2}

\newcommand\GZ{\ensuremath{\GL_2(\Z)}\xspace}
\newcommand\Sf{\ensuremath{{S_f}}\xspace}
\newcommand\OSf{\ensuremath{\mathcal{O}_\Sf}\xspace}

\newcommand\BS{\ensuremath{S}\xspace}

\newcommand\OS{\ensuremath{{\mathcal{O}_\BS}}\xspace}

\newcommand\m{\ensuremath{\mathfrak{m}}\xspace}

\newcommand\map[4]{\ensuremath{\begin{array}{ccc}#1&\lra&#2\\#3&\mapsto&#4\end{array}}}

\newcommand\bq{\begin{equation}}
\newcommand\eq{\end{equation}}

\newtheorem{proposition}{Proposition}[section]
\newtheorem{theorem}[proposition]{Theorem}
\newtheorem{corollary}[proposition]{Corollary}
\newtheorem{example}[proposition]{Example}

\newtheorem{lemma}[proposition]{Lemma}

\theoremstyle{remark}
\newtheorem{remark}[proposition]{Remark}
\usepackage[all,cmtip]{xy}

\newcommand{\U}{\mathcal{U}}
\renewcommand{\v}{\mathcal{V}}
\newcommand{\arisom}{\stackrel{\sim}{\ra}}

\usepackage{fullpage}

\newenvironment{definition}{\vspace{2 ex}{\noindent{\bf Definition. }}}{\vspace{2 ex}}
\newenvironment{notation}{\vspace{2 ex}{\noindent{\bf Notation. }}}{\vspace{2 ex}}

\newcommand\GGZ{\ensuremath{\GL_2(\Z)\times \GL_3(\Z)}\xspace}

\title{Parametrizing quartic algebras over an arbitrary base}

\author{Melanie Matchett Wood\thanks{mwood@math.stanford.edu}\\
Stanford University\\
Department of Mathematics\\
Building 380, Sloan Hall\\
Stanford, California 94305\\
USA}

\begin{document}

\maketitle

\abstract{We parametrize quartic commutative algebras over any base ring or scheme (equivalently finite, flat degree four
$S$-schemes), with their cubic resolvents, by pairs of ternary quadratic forms over the base.  This generalizes Bhargava's parametrization of quartic rings with their cubic resolvent rings over $\Z$ by pairs of integral ternary quadratic forms, as well as Casnati and Ekedahl's construction of Gorenstein quartic covers by certain rank 2 families of ternary quadratic forms.  We give a geometric construction of a quartic algebra from any pair of ternary quadratic forms, and prove this construction commutes with base change and also agrees with Bhargava's explicit construction over $\Z$.}

\section{Introduction}\label{intro}

\subsection{Definitions and main result}
A \emph{$n$-ic algebra} $Q$ over a scheme $S$  is an $\OS$-algebra $Q$ that is a locally free rank $n$ $\OS$-module, or equivalently $\Spec Q$ is a finite, flat degree $n$ $S$-scheme.
For $n=3,4$, we call such algebras \emph{cubic} and \emph{quartic} respectively.
Given a quartic algebra, we can define a cubic resolvent which is a model over $S$ of the classical cubic resolvent field of a quartic field.
For a quartic algebra $Q$ over $S$, \emph{a cubic resolvent}
$C$ of $Q$ is a cubic algebra $C$ over $S$, with a quadratic map $\phi : Q/\OS \ra C/\OS$ and an isomorphism $\delta : \wedge^4 Q \arisom \wedge^3 C$,
such that for any sections $x,y$ of $Q$ we have $\delta(1\wedge x\wedge y\wedge xy)=1 \wedge \phi(x)\wedge \phi(y) $ and also $R$ is the cubic algebra corresponding to $\Det (\phi)$ (see Section~\ref{S:resolvent} for more details).
 An isomorphism of a pair $(Q,C)$ is given by isomorphisms of the respective algebras that respect $\phi$ and $\delta$.

A \emph{double ternary quadratic form} over $S$ is a locally free rank 3 $\OS$-module
$W$, a locally free rank 2 $\OS$-module $U$, and a global section $p\in \Sym^2 W \tesnor U$, and an isomorphism $\wedge^3 W \tensor \wedge^2 U \stackrel{\sim}{\ra} \OS$ which is called an \emph{orientation}.
An isomorphism of double ternary quadratic forms $(W,U,p)$ and $(W',U',p')$ is given by isomorphisms $W\arisom W'$ and $U\arisom U'$ that
send $p$ to $p'$ and respect the orientations.

The main theorem of this paper is the following.
\begin{theorem}\label{T:genquartic}
There is an isomorphism between the moduli stack for quartic algebras with cubic resolvents and the moduli stack for double ternary quadratic forms.
In other words, for a scheme $S$ there is an equivalence between the category of quartic algebras with cubic resolvents and the category of
double ternary quadratic forms (with morphisms given by isomorphisms in both categories), and this natural equivalence commutes with base change in $S$.
\end{theorem}

The moduli stack of double ternary quadratic forms is simply $[\A^{12}/\Gamma]$, where $\Gamma$ is the sub-group scheme of $\GL_2\times\GL_3$ of elements
$(g,h)$ such that $\det(g)\det(h)=1$.
In particular, we have a parametrization of quartic algebras with cubic resolvents.
\begin{corollary}
  Over a scheme $S$, there is a bijection between isomorphism classes of
double ternary quadratic forms over $S$ and isomorphism classes of pairs $(Q,C)$ where $Q$ is a quartic algebra over $S$ and 
$C$ is a cubic resolvent of $Q$. 
\end{corollary}

 \begin{remark}
The geometric language of this paper makes it more natural to work over a scheme $S$,
but all of our work includes the case $S=\Spec R$, in which case we are simply working over a ring $R$.
The reader mainly interested in a base ring can replace $\OS$ with $R$
and ``global section'' with ``element''
throughout the paper.   
\end{remark}

\subsection{Background and previous work}
It has been known since the work of Delone and Faddeev \cite{DF} (see also Section~\ref{S:cubic} of this paper, \cite{DH}, \cite{GGS}, and \cite{HCL3}) that cubic rings are parametrized by
binary cubic forms.  A \emph{cubic ring} is a ring whose additive structure is a free rank 3 \Z-module, and a
\emph{binary cubic form} is a polynomial $f=ax^3+bx^2y+cxy^2+dy^3$ with $a,b,c,d\in \Z$.  
Cubic rings, up to isomorphism, are in natural discriminant-preserving bijection with \GZ-classes of binary cubic forms.  If we prefer to think geometrically, a cubic ring is a finite, flat degree three cover of $\Spec \Z$.  
A parametrization analogous to that of Delone and Faddeev \cite{DF} was proven by Miranda \cite{Miranda} for finite, flat degree three
covers of an irreducible scheme over an algebraically closed field of characteristic not 2 or 3.  
Though these correspondences were originally given by writing down a multiplication table for the cubic ring
(or sheaf of functions on the cubic cover), when $f$ is a non-zero integral binary cubic form, the associated cubic ring 
is simply the ring of global functions of the subscheme of $\P^1_\Z$ cut out by $f$ (see \cite{Delcubic}, \cite[Theorem 2.4]{binarynic}, \cite{CasI}).

In this paper, we study quartic (commutative) algebras, or equivalently, finite, flat degree four covers of a base scheme.  Casnati and Ekedahl \cite{CasI}
found that finite, flat degree four Gorenstein covers of an integral base scheme are given by global sections of certain double ternary quadratic forms, with a codimension condition on the section at every point of the base.  
Recently, quartic algebras over $\Z$ have been parametrized
by Bhargava \cite{HCL3}. More precisely, Bhargava proved that
isomorphism classes of
pairs $(Q,C)$, where $Q$ is a quartic ring (i.e. isomorphic to $\Z^4$ as a $\Z$-module)  and $C$ is a cubic resolvent of $Q$,
are in natural bijection with $\GGZ$-classes of pairs of integral ternary quadratic forms.  
(We could view a pair of ternary quadratic forms over \Z
as a double ternary quadratic form 
$\sum_{1\leq i\leq j\leq3} a_{ij} x_i x_jy$+$\sum_{1\leq i\leq j\leq3} b_{ij} x_i x_jz$.)
In \cite{HCL3}, Bhargava introduced cubic resolvents as models
of the classical cubic resolvent field of a quartic field. 
All quartic rings over \Z have at least one cubic resolvent, and 
many quartic rings (for example, maximal quartic rings over \Z) have a unique cubic resolvent \cite[Corollary 4]{HCL3}.  
This has allowed Bhargava \cite{countquartic} 
to count asymptotically the number of $S_4$ number fields of discriminant less than $X$ 
(as well as the number of orders in $S_4$ number fields). 
Casnati \cite{CasIII} has also given a construction of a finite, flat degree three ``discriminant cover'' corresponding to a finite, flat degree four Gorenstein cover of an integral scheme over an algebraically closed field of characteristic not equal to 2, but since he was only considering quartic covers that turn out to have unique cubic resolvents, the importance of the cubic resolvent to the moduli problem was not apparent.
 Bhargava \cite{HCL3} realized that to obtain a nice parametrization of quartic rings over $\Z$, one must parametrize them along with
their cubic resolvents.

In this paper, we generalize Bhargava's work \cite{HCL3} from $\Z$ to an arbitrary scheme, and Casnati and Ekedahl's work \cite{CasI} from the case of Gorenstein covers and special forms to all quartic covers and forms, as well as to an arbitrary base scheme.  Moreover, we prove our correspondence between quartic algebras with resolvents and double ternary quadratic forms commutes with base change.

Bhargava, in \cite{HCL3}, describes the relationship between quartic rings with cubic resolvents 
and  pairs of ternary quadratic forms by giving the multiplication tables
for the quartic and cubic rings explicitly in terms of the coefficients of the forms.
In this paper, we give a geometric, coordinate-free description of a quartic ring $Q$ given by a pair of integral ternary quadratic forms.
For the nicest forms, the pair of ternary quadratic forms gives a pencil of conics in $\P^2_\Z$ and the quartic
ring is given by the global functions of the degree four subscheme cut out by the pencil.  
In Section~\ref{S:geom}, we give a global, geometric, coordinate-free construction over a quartic algebra from a double ternary quadratic form over any scheme $S$.
The construction that works for all forms is taking the degree $0$ hypercohomology of the Koszul complex of the double ternary quadratic form.  This
agrees with the intuitive geometric description given above for nice cases, but unlike the above description, always gives a quartic algebra.
Casnati and Ekedahl \cite{CasI} have given an analogous gemetric construction over an arbitrary scheme in the case when the quartic algebra is Gorenstein.  
Deligne, in a letter \cite{Delquartic} to Bhargava, gives an analogous geometric construction when the generic conic 
in the pencil is non-singular over each geometric point, and proves that it extends (without giving a geometric construction in the extended case)
to all pairs of ternary quadratic forms.
The geometric construction in this paper works for all double ternary quadratic forms, for example
when the form is idenitcally 0 in some fiber, when the conics given by the ternary quadratic forms share a component,
or even when both forms are identically 0!  

In Section~\ref{S:local}, we explain how the quartic algebra associated to a double ternary quadratic form over $S$ 
can be defined locally in terms of the multiplication tables given in Bhargava's work \cite{HCL3}, and prove that these
constructions agree.  The calculations showing this agreement are not straightforward and are given in Section~\ref{SameOnUniversal}.

In Section~\ref{S:cubic}, we review the parametrization of cubic algebras.
This is not only motivation for our study of quartic algebras, but also is important background for the results
in this paper because the cubic resolvent $C$ is a cubic algebra.  In
 Section~\ref{S:resolvent} we give the definition of a cubic resolvent in more detail.
 In Section~\ref{S:constcubic}, we give the construction of a cubic resolvent 
  from a double ternary quadratic form.
In Section~\ref{S:Main}, we prove Theorem~\ref{T:genquartic}.

\begin{notation}
 If $\mathcal{F}$ is a sheaf, we use $s\in\mathcal{F}$ to denote that $s$ is a global section of
$\mathcal{F}$.   If $V$ is a locally free $\OS$-module, we use $V^*$ to denote
the $\OS$-module $\mathcal{H}om_\OS (V, \OS)$.  We use $\Sym^n V$ to denote the usual
quotient of $V^{\tensor n}$, and $\Sym_n V$ to denote the submodule of symmetric elements of $V^{\tensor n}$.
Note that when $V$ is locally free we have $\Sym_n V=(\Sym^n V^*)^*$ (see Lemma~\ref{dualsyms}). We define $\P(V)=\Proj \Sym^* V$.

 Normally, in the language of algebra,
 one says that an $R$-module $M$ is locally free of rank $n$
if for all prime ideals $\wp$ of $R$, the localization $M_\wp$ is free of rank $n$.
However, if we have a scheme $S$ and an $\OS$-module $M$, we normally
say that $M$ is locally free of rank $n$ if on some open cover of $S$ it is free of rank $n$;
in the algebraic language this is equivalent to saying that for every prime ideal $\wp$ of $R$, 
there is an $f\in R \setminus \wp$ such that the)
 localization $M_f$ is free of rank $n$.  
In this paper we shall use the geometric sense of the term \emph{locally free of rank $n$}.
The geometric condition of locally free of rank $n$ is equivalent to being finitely 
generated and having the algebraic condition of locally free of rank $n$. 
\end{notation}

\section{The parametrization of cubic algebras}\label{S:cubic}
In this section, we review the parametrization of cubic algebras.  
A \emph{binary cubic form} over a scheme $S$ is a locally free rank 2 $\OS$-module
$V$ and an  $f\in\Sym^3 V \tensor  \wedge^2 V^*$.  An isomorphism of binary cubic forms $(V,f)\isom(V,f')$
is given by an isomorphism $V\isom V'$ that takes $f$ to $f'$.  
(Normally, we would call these \emph{twisted binary cubic forms} but since they are 
the only binary cubic forms in this paper, we will use the shorter name for simplicity.)
Of course, if $V$ is
the free rank 2 $\OS$-module $\OS x \oplus \OS y$, then the binary cubic
forms $f\in\Sym^3 V \tensor  \wedge^2 V^*$ are just polynomials $(ax^3+bx^2y+cxy^2+dy^3)\tensor(x\wedge y)^*$, where 
$a,b,c,d\in\OS$.

Over an arbitrary base, Deligne wrote a letter \cite{Delcubic} to the authors of \cite{GGS}
giving the following theorem.

\begin{theorem}[\cite{Delcubic},\cite{Poonen}]\label{cubicbij}
There is an isomorphism between the moduli stack for cubic algebras and the moduli stack for binary cubic forms.
That is, there is an equivalence of categories between the category of cubic algebras over $S$ where morphisms are given by
isomorphisms and the category of binary cubic forms over $S$ where morphisms are given by isomorphisms, and this
equivalence commutes with base change in $S$.
Thus, over a scheme $S$, there is a bijection between isomorphism classes of cubic algebras and isomorphism classes of binary cubic forms.
 If a cubic algebra $C$ corresponds to a binary cubic form $(V,f)$, then as $\OS$-modules,
we have $C/\OS\isom V^*$.
\end{theorem}

This theorem is also proven in \cite[Proposition 5.1]{Poonen}.
Miranda \cite{Miranda} gives the bijection between isomorphism classes over a base which is an irreducible scheme over an algebraically closed field
of characteristic not equal to 2 or 3.
Also, this isomorphism of stacks is studied and proven as part of a series of such isomorphisms involving binary forms of any degree in \cite{binarynic}.

In \cite[Footnote 3]{HCL2}, the following algebraic, global, coordinate free description of the construction 
of a binary cubic form from a cubic algebra
is mentioned.
  Given a cubic algebra $C$, we can define an $\OS$-module $V=(C/\OS)^*$.  (Note that $V$ is a locally free rank $2$ $\OS$-module,
see e.g. \cite[Lemma 1.3]{Voight}.)
We can then define an $\OS$-module homomorphism
$\Sym_3 C/\OS \ra \wedge^2 C/\OS$ given by $ xyz \mapsto x \wedge yz$.  One can check that this map
is well-defined, and so it gives a binary cubic form $f\in (\Sym_3 C/\OS)^* \tesnor  \wedge^2 C/\OS
\isom \Sym^3 V \tesnor \wedge^2 V^*$.
Deligne in \cite{Delcubic} gives a different, geometric construction in
the case when $C$ is Gorenstein and then argues that the construction extends across the non-Gorenstein locus.

It is often useful to also have the following local, explicit version of the construction.
Where $C$ is a free $\OS$-module, we can choose a basis $1,\omega,\theta$ for $C$ and then shift $\omega$ and
$\theta$ by elements of $\OS$
so that $\omega\theta\in\OS$.  Then, the associative law implies that we have a multiplication table
\begin{align}
 \omega\theta&= -ad \notag\\ 
\omega^2 &= -ac +b\omega -a \theta\\
\theta^2&= -bd+d\omega-c\theta, \notag
\end{align}
where $a,b,c,d \in\OS$.  
Let $x,y$ be the basis of $V$ dual to $\omega,\theta$. 
Then we can define a form $(ax^3+bx^2y+cxy^2+dy^3)\tensor(x\wedge y)^*\in\Sym^3 V\tensor \wedge^2 V^*$.
We can check that if we pick another basis $1,\omega',\theta'$ (also
normalized so that $\omega'\theta'\in\OS$) and another corresponding $x'$ and $y'$ we would define the same
form in $\Sym^3 V\tensor \wedge^2 V^*$.
Thus the form is defined everywhere locally in a way that agrees on overlapping open sets, and we have constructed
a global binary cubic form $(V,f)$.  

One construction of a cubic algebra from a form (i.e. the inverse to the above construction) simply gives the cubic algebra locally by the above multiplication table.  This gives the bijection locally in terms of bases with explicit formulas.  However, it is hard to see where the formula for the multiplication table came from or why the local constructions are invariant under change of basis.
The following global description is given by Deligne in his letter \cite{Delcubic}.
Given a binary form $f\in \Sym^3 V\tensor \wedge^2 V^*$ over a base scheme $\BS$,
the form $f$ determines a subscheme $\Sf$ of $\P(V)$.
Let $\pi: \P(V) \ra S$.  Let $\mathcal{O}(k)$ denote
the usual sheaf on $\P(V)$ and $\OSf(k)$ denote the pullback
of $\mathcal{O}(k)$ to
\Sf. Then we can define the $\mathcal{O}_\BS$-algebra by the hypercohomology
\begin{equation}
C:=H^0 R\pi_* \left(\mathcal{O}(-3)\tensor \pi^*\wedge^2 V \stackrel{f}{\ra} \mathcal{O}\right),
\end{equation}
where $\mathcal{O}(-3)\tensor \pi^*\wedge^2 V \stackrel{f}{\ra} \mathcal{O}$ is a complex in degrees -1 and 0.
 The product on $C$ is 
given
by the product on the Koszul complex $\mathcal{O}(-3 ) \tensor \pi^*\wedge^2 V \stackrel{f}{\ra} \mathcal{O}$
with itself and the $\mathcal{O}_\BS$-algebra structure is induced from the map of $\mathcal{O}$
as a complex in degree 0 to the complex $\mathcal{O}(-3) \stackrel{f}{\ra} \mathcal{O}$ (see Appendix Section~\ref{Appendix2} for more details on the inheritance of the algebra structure).
(Note that $H^0R \pi_*(\mathcal{O})=\mathcal{O}_\BS$.)  

Given a map of schemes $X \stackrel{\pi}{\ra} S$, the construction of global functions of $X$ relative to $S$
is just the pushforward $\pi_* (\mathcal{O}_X)$.  So the natural notion of global functions of $S_f$ relative
to $S$ would be $\pi_*$ of $\OSf$.  We have that $\OSf=\OS/f(\mathcal{O}(-3)\tensor \pi^*\wedge^2 V)$.
  When $f$ is injective, then
$\OSf=\OS/f(\mathcal{O}(-3)\tensor \pi^*\wedge^2 V)$ as a complex in degree 0 has the same hypercohomology as
$\mathcal{O}(-3)\tensor \pi^*\wedge^2 V \stackrel{f}{\ra} \mathcal{O}$ as a complex in degrees -1 and 0.
Thus we see when $f$ is injective that $C$ is just $\pi_*(\mathcal{O}_\Sf)$.
 When $f$ gives an injective map and $S=\Spec R$ then $C$ is just the ring
of global functions of $\Sf$.  Unfortunately, this simpler construction does not give a cubic algebra when $f\equiv0$.
When $f\equiv0$, then $\Sf=\P^1$ and the global functions are a rank 1 $\OS$-algebra, i.e. $\OS$ itself.
 Hypercohomology is exactly the machinery we need to naturally extend the construction to all $f$.

\section{Cubic resolvents}\label{S:resolvent}

 We now give the definition of a cubic resolvent, given first in
\cite[Definition 20]{HCL3} over $\Z$. The definition might seem
complicated at first, but we will explain each aspect of it.

\begin{definition}
 Given a quartic algebra $Q$ over a base scheme $S$, a cubic resolvent
$C$ of $Q$ is
\begin{itemize}
 \item a cubic algebra $C$ over $S$ 
 \item a quadratic map $\phi : Q/\OS \ra C/\OS$, and
 \item an isomorphism $\delta : \wedge^4 Q \arisom \wedge^3 C$
(or equivalently $\bar{\delta}: \wedge^3 Q/\OS \arisom \wedge^2 C/\OS$), which we call the \emph{orientation}
\end{itemize}
such that
\begin{enumerate}
 \item for any open set $U\subset S$ and for all $x,y\in Q(U)$, 
we have $\delta(1\wedge x\wedge y\wedge xy)=1 \wedge \phi(x)\wedge \phi(y) $
 \item $C$ is the cubic algebra corresponding to $\Det (\phi)$.
\end{enumerate}
\end{definition}

Note that $Q/\OS$ and $C/\OS$ are locally free $\OS$-modules of ranks $3$ and $2$, respectively (see e.g. \cite[Lemma 1.3]{Voight}).
A \emph{quadratic map} from $A$ to $B$
is given by an $\OS$-module homomorphism $\Sym_2 A \ra B$ evaluated on the diagonal (see the Appendix Section~\ref{SS:degkmaps}).
(In \cite[Proposition 6.1]{binquad} it is shown this is equivalent to the more classical notion of a quadratic map.)  The map $\phi$ models the map from quartic fields
to their resolvent fields given by $x\mapsto xx' + x''x'''$, where
$x,x',x'',x'''$ are the conjugates of an element $x$.  In \cite[Lemma 9]{HCL3}
it is shown that condition 1 above holds for such classical resolvent maps,
and it turns out that condition 1 is the key property of resolvent maps
that allows them to be useful in the parametrization of quartic algebras.
So the definition of resolvent allows all quadratic maps that have this key property.

Another important property of the cubic resolvent over $\Z$ is that the discriminant of
the cubic resolvent is equal to the discriminant of the quartic ring.  In \cite{HCL3}, this is a crucial part of the definition
of a cubic resolvent over the integers.  With the above formulation
of the definition of a cubic resolvent, the equality of discriminants follows as a corollary of properties 1 and 2.  However, since the discriminant of an algebra $R$
of rank $n$
lies in $(\wedge^n R)^{\tesnor -2}$, we need the orientation isomorphism to even
state the question of the equality of discriminants.
The orientation is a phenomenon that it is hard to recognize the importance of over $\Z$ because
$\GL_1(\Z)$ is so small, however it appears in Bhargava's \cite{HCL3} choice of bases for a quartic ring and its cubic resolvent.

The quadratic map $\phi$ is equivalent to a double ternary quadratic form in the module
$\Sym^2 (Q/\OS)^*\tensor C/\OS$.   The determinant of
a double ternary quadratic form is given by a natural cubic map from $\Sym^2 W \tesnor V$ to $(\wedge^3 W)^{\tensor 2}\tesnor\Sym^3 V  $.  
We have a natural cubic determinant map from $\Sym^2 W $ to $(\wedge^3 W)^{\tensor 2}$. For free $W$ and an element of $\Sym^2 W $ represented by the matrix
$$
A=\begin{pmatrix} a_{11} & \frac{a_{12}}{2} & \frac{a_{13}}{2} \\
\frac{a_{12}}{2} & a_{22} & \frac{a_{23}}{2} \\
\frac{a_{13}}{2} & \frac{a_{23}}{2} & a_{33} 
\end{pmatrix},
$$
the map is given by the polynomial $4\Det(A)$, and since this is invariant under $\GL_3$ change of basis, 
it defines a determinant map
for all locally free $W$.  
We use $2$'s in the denominator of our expression for $A$ because it allows us a convenient way to express the polynomial $4\Det(A)$, but note that
the polynomial given by $4\Det(A)$ does not have any denominators, and thus we do not need to require that 2 is invertible to construct the determinant of a double ternary quadratic form.
We can extend to a cubic determinant map from 
$\Sym^2 W \tesnor V$ to $(\wedge^3 W)^{\tensor 2}\tesnor\Sym^3 V $
by using the elements of $V$ as coefficients (see Appendix Section~\ref{SS:degkcoeff}).
Thus the determinant of $\phi$ lies
in   $(\wedge^3 Q/\OS)^{\tensor -2}\tesnor\Sym^3 (C/\OS)$, which is isomorphic to
 $(\wedge^2 C/\OS)\tesnor\Sym^3 (C/\OS)^* $ by the orientation isomorphism
(see also Corollary~\ref{C:cubic} in the Appendix).  From Theorem~\ref{cubicbij}, we have that $C$ corresponds to a global section
of $(\wedge^2 C/\OS)\tesnor\Sym^3 (C/\OS)^* $, and condition 2 above is that $C$ corresponds to the section $\Det(\phi)$.

When we speak of a pair $(Q,C)$ of a quartic algebra $Q$ and a cubic resolvent $C$ of $Q$, the 
maps $\phi$ and $\delta$ are implicit.  An isomorphism of pairs is
given by isomorphisms of the respective algebras that respect $\phi$ and $\delta$.

\section{The geometric construction}\label{S:geom}
In this section, we will construct a quartic algebra from a double ternary quadratic form $p\in \Sym^2 W \tesnor U$ over a base $S$.
We consider the map $\pi :\P(W) \ra S$, and the usual line bundles $\mathcal{O}(k)$ on $\P(W)$.  We can view $p$ as a two dimensional family of quadratic forms on $\P(W)$
(the two dimensions being given by $U$).  More precisely, since $p$ is equivalent to a map
$U^* \ra \Sym^2 W$,  we have a naturally induced map $\pi^* U^* \ra \mathcal{O}(2)$, which is equivalent to a map
$p_1 : \pi^* U^* \tensor  \mathcal{O}(-2)\ra \mathcal{O}$.  The image of $p_1$ is functions that are zero on the space cut out by the forms of $p$.  The regular functions on the scheme cut out by $p$ are just given by $\mathcal{O}/\im(p_1)$.  From $p$ we can construct one more map to make the Koszul complex of $p$, given as follows
$$
\mathcal{K}_p : \quad \wt \pi^* U^* \tensor \mathcal{O}(-4) \xrightarrow{p_2 } \pi^* U^* \tensor \mathcal{O}(-2)\xrightarrow{p_1} \mathcal{O}.
$$

The complex $\mathcal{K}_p$ has $\mathcal{O}$ in degree 0, and the other two terms in degrees $-1$ and $-2$.
We can construct $p_2$ similarly to $p_1$ since $p$ is also equivalent to a map $\wt U^* \tensor U  \ra \Sym^2 W$.
(Recall $\wt U^* \tensor U \isom U^*$; see Lemma~\ref{L:rk2}.)   One can read about the construction of all the maps in the Koszul complex in \cite[Appendices A2F and A2H]{GeomSyz}.

\begin{example}\label{freeexample}
Suppose $U$ is free with basis $x,y$, and dual basis $\dot{x}$ and $\dot{y}$.  Then we can write $p= f_1 \tesnor x +f_2\tensor y$.
The map $p_1$ just sends $\dot{x}\tesnor g \mapsto f_1 g$ and $\dot{y}\tesnor g \mapsto f_2 g$.  
We can write how $p_1$ acts on a general element as
$ a \tesnor g \mapsto  gp(a)$, where $p$ acts on an element of $U^*$ by evaluating the $U$ components of $p$
at the given element of $U^*$.
The map $p_2$
sends $\dot{x} \wedge \dot{y} \tesnor g \mapsto gf_1\tesnor \dot{y} -gf_2 \tensor \dot{x}$.  We can write how $p_2$ acts on a general element as
$ a\wedge b \tesnor g \mapsto b\tesnor gp(a) - a \tensor g p(b)$.
From this we see that $\mathcal{K}_p$ is a complex.
\end{example} 

For sufficiently nice $p$ the Koszul complex will be exact in all places except the last and thus give a resolution of $\mathcal{O}/\im(p_1)$.  For example, this it true when $p$ is the universal double ternary quadratic form over the polynomial ring in twelve variables.  
In this well-behaved case, $p$ will cut out four (relative) points in $\P(W)$ (i.e. a finite, flat degree four $S$-scheme) and the pushforward of the global functions of those points
will give us a quadratic algebra over the base $S$.

When the Koszul complex of $p$ is not a resolution, instead of taking the pushforward of the global functions of the scheme cut out by $p$, we will take the $0$th hypercohomology of the complex $\mathcal{K}_p$.
We define $Q_p$ to be $H^0 R \pi_* (\mathcal{K}_p)$, where $R \pi_*$ denotes the pushforward of the complex in the derived category.
Alternatively, we can view the construction as the hypercohomological
derived functor of $\pi_*$, where the hypercohomology is necessary since we
 are operating on a complex and not just a single sheaf.
If $p$ is nice enough that its Koszul complex $\mathcal{K}_p$ is a resolution of $\mathcal{O}/\im(p_1)$, then $Q_p$ will just be $\pi_*(\mathcal{O}/\im(p_1))$.  However, what is convenient about the hypercohomology construction is that $Q_p$ will be a quartic algebra 
even when $\mathcal{K}_p$ is not a resolution (as we'll see in Section~\ref{SS:module}).  So far we have constructed $Q_p$ as an $\OS$-module,
however, the Koszul complex has a natural differential graded algebra structure, and that gives the cohomology an inherited algebra structure (see Appendix Section~\ref{Appendix2} for more details on the inheritance of the algebra structure).  
The map from $\mathcal{O}$ as a complex in degree 0 to the complex $\mathcal{K}_p$ induces a map from $H^0R \pi_*(\mathcal{O})=\mathcal{O}_\BS\ra Q_p$. 
This gives $Q_p$ the structure of an $\OS$-algebra.

\subsection{Examples when $\mathcal{K}_p$ is not a resolution}
When constructing the cubic algebra from a binary cubic form, we took
$$
H^0 R \pi_* ( \mathcal{O}(-3) \xrightarrow{f} \mathcal{O} )
$$
on $\P(V)$, which, as long as the cubic form $f$ gives an injective map above is the same as
$ \pi_* (\mathcal{O}/\im f)$.  For example, when the base $S$ is integral, whenever $f\not\equiv 0$
then $\mathcal{O}(-3) \xrightarrow{f} \mathcal{O}$ is injective.  However, when $f\equiv0$, of course
$\mathcal{O}(-3) \xrightarrow{f} \mathcal{O}$ is not injective, and $H^0 R \pi_* ( \mathcal{O}(-3) \xrightarrow{f} \mathcal{O} )$
is not the same as $\pi_* (\mathcal{O}/\im f )$.  When $f\equiv0$, the latter is an $\OS$-module of rank 1.

Again, when constructing our quartic algebra as $H^0 R \pi_* (\mathcal{K}_p)$, if $p\equiv0$ the complex
will not be a resolution and $H^0 R \pi_* (\mathcal{K}_p)$ won't agree with
$ \pi_* (\mathcal{O}/\im p_1 )$.  
This is the case when both ``conics'' are given by the 0 form. 
However, even over an integral base, there are now more situations on which
the complex $\mathcal{K}_p$ is not a resolution.  
The geometric constructions of Casnati and Ekedahl \cite{CasI} and Deligne \cite{Delquartic} for certain nice quartic algebras 
are in cases when $ \pi_* (\mathcal{O}/\im p_1 )$ simply gives the quartic algebra.

We now give several examples in which $\mathcal{K}_p$ is not a resolution. 

\begin{example}
 Let $p\equiv 0$.  Then $Q_p = \OS \oplus W^*$, with the multiplication given by
$W^* \tensor_{\OS} W^* \ra 0$.
\end{example}

Let $U$ be free with the notation of Example~\ref{freeexample}.

\begin{example}
  If $f_2\equiv 0$, then 
$Q_p = \OS \oplus W^*$, with the multiplication given by
$W^* \tensor_{\OS} W^* \ra 0$.
\end{example}

Now, let $W$ be free on $w_1,w_2,w_3$.

\begin{example}
If $f_1=w_1w_2$ and $f_2=w_1w_3$, then $Q_p\isom \OS \oplus \OS[z_1,z_2]/(z_1,z_2)^2$.
This is the case where the two conics share a linear component, and the
pencil of second lines all go through a  point not on the shared line.
\end{example}

\begin{example}
If $f_1=w_1w_2$ and $f_2=w_1^2$, then $Q_p\isom \OS[z_1,z_2]/(z_1^3,z_1z_2,z_2^2)$.
This is the case where the two conics share a linear component, and the
pencil of second lines all go through a  point on the shared line.
\end{example}

\begin{example}
 If $f_1=w_1^2+w_1w_3$ and $f_2=w_2^2+w_2w_3$, then $\mathcal{K}_p$ \emph{is} a resolution
and $Q_p\isom R:=\OS\oplus\OS\oplus\OS\oplus\OS$.
However, unlike in the case of binary cubic forms, we can change $p$ in just one closed fiber
and $\mathcal{K}_p$ will no longer be a resolution.  For simplicity, let $S=\Spec \Z$, and let
$q$ be a prime.  Then if $f_1=q(w_1^2+w_1w_3)$ and $f_2=q(w_2^2+w_2w_3)$, the global functions
of the subscheme cut out by $p$ 
are isomorphic to $\Z \oplus pR \sub R$ (a quartic \Z-algebra) but $Q_p\isom \Z \oplus p^2R \sub R$. 
\end{example}

\subsection{Module structure of $Q_p$}\label{SS:module}

In this section, we determine the $\OS$-module structure of $Q_p$.
We consider the short exact sequence of complexes $O\ra \mathcal{A} \ra \mathcal{K}_p \ra \mathcal{D} \ra 0$, where
\begin{equation*}
\xymatrix @R=.1in {
 \mathcal{A}: & 0 \ar[r]& \pi^* U^* \tensor \mathcal{O}(-2)\ar[r]^-{p_1}& \mathcal{O}\\
\mathcal{K}_p:&  \wt \pi^* U^* \tensor \mathcal{O}(-4) \ar[r]^{p_2 }& \pi^* U^* \tensor \mathcal{O}(-2) \ar[r]^-{p_1}& \mathcal{O}\\
\mathcal{B}:&  \wt \pi^* U^* \tensor \mathcal{O}(-4) \ar[r]& 0 \ar[r]& 0.}
\end{equation*}
From this short exact sequence we obtain a long exact sequence of hypercohomology sheaves
on $S$, of which we consider the following part
$$
\xymatrix@R=.1in{
 {H ^{-1} R\pi_*} (\mathcal{B}) \ar[r] &
{H ^0 R\pi_*} ( \mathcal{A}) \ar[r] & {H ^0 R\pi_*} (\mathcal{K}_p) \ar[r] \ar@{=}[d] & {H ^0 R\pi_*} (\mathcal{B} ) \ar[r] & {H ^1 R\pi_*} ( \mathcal{A}).\\
& & Q_p
} 
$$
This sequence will allow us the determine the modules structure of $Q_p$ once we compute the other terms.
It is natural to shift the term in $\mathcal{B}$ to degree 0 and obtain
$$
\xymatrix@R=.15in{
 R ^{1} \pi_* (\wt \pi^* U^* \tensor \mathcal{O}(-4)) \ar[r] \ar@{=}[d]& {H ^0R \pi_*}  (\mathcal{A})\ar[r] & Q_p\ar[r] &R ^2 \pi_* (\wt \pi^* U^* \tensor \mathcal{O}(-4) ) \ar[r]\ar@{=}[d] & { H ^1 R\pi_*}  (\mathcal{A}).\\
0 & & & W^* \tesnor \wedge^3 W^* \tensor \wedge^2 U^* \ar[d]^{\isom}&\\
& & &W^* 
}
$$
 We can analyze the $\mathcal{A}$ terms by putting the complex $\mathcal{A}$
in its own short exact sequence of complexes $0\ra\mathcal{D} \ra \mathcal{A} \ra \mathcal{E}\ra 0$, given by the following
\begin{equation*}
\xymatrix@R=.1in{
 \mathcal{D}:& 0\ar[r] &  \mathcal{O}\\
\mathcal{A} :& \pi^* U^* \tensor  \mathcal{O}(-2)\ar[r]^-{p_1}&  \mathcal{O}\\
\mathcal{E} :&\pi^* U^* \tensor \mathcal{O}(-2)\ar[r] & 0.}
\end{equation*}
Taking the long exact sequence for this short exact sequence of complexes gives
$$
{H ^{-1}R \pi_*} (\mathcal{E})\ra {H ^0 R\pi_*} ( \mathcal{D})\ra {H ^0 R\pi_*} (\mathcal{A})\ra {H ^0R \pi_*} (\mathcal{E} ) \ra {H ^1R \pi_*} ( \mathcal{D})\ra {H ^1R \pi_*} (\mathcal{A})\ra {H ^1R\pi_*} (\mathcal{E}),
$$
or
$$
\xymatrix@R=.1in{
 R ^{0} \pi_* (\pi^* U^* \tensor \mathcal{O}(-2))\ar[r]\ar@{=}[d] &  R ^0 \pi_* ( \mathcal{O})\ar[r] \ar@{=}[d]& {H ^0R \pi_*} (\mathcal{A})\ar[r] &  R ^1 \pi_* (\pi^* U^* \tensor \mathcal{O}(-2) ) \ar@{=}[d]\\
0 & \OS & & 0
}
$$
and $$
\xymatrix@R=.1in{
R ^1 \pi_* ( \mathcal{O})\ar[r]\ar@{=}[d] & {H ^1 R\pi_*} (\mathcal{A})\ar[r]&  R ^2 \pi_* (\pi^* U^* \tensor \mathcal{O}(-2) ).\ar@{=}[d] \\
0 & & 0}
$$
Thus, we conclude that 
${H ^0 R\pi_*} (\mathcal{A})\isom \OS$ and ${H ^1 R\pi_*} (\mathcal{A})=0$.

Going back to our original long exact sequence, we have
$$
 0 \ra \OS \ra Q_p \ra  W^*  \ra 0.
$$
This proves that $Q_p$ is a locally free rank 4 $\OS$-module.  Also, it gives us the 
necessary map $\OS \ra Q_p$ for our algebra to have a unit.  (We can check this map respects the algebra structures
because it is induced from the map of complexes $\mathcal{D}\ra \mathcal{K}_p$ that respects the differential
graded algebra structures on $\mathcal{D}$ and $\mathcal{K}_p$.)

\begin{theorem}
 The construction of $Q_p$ commutes with base change in $S$.
\end{theorem}
\begin{proof}
 To prove this theorem, we will compute all of the cohomology of $\mathcal{K}_p$.
The complex $\mathcal{K}_p$ has no cohomology in degrees other than 0.
We have $R^k\pi(\mathcal{O}(-4))=0$ for $k\ne 2$, and
$R^k\pi_*(\mathcal{O}(-2))=0$ for all $k$,
and $R^k\pi_*(\mathcal{O})=0$ for $k\ne 0$.  
Thus $H^kR\pi_*(\mathcal{K}_p)=0$ for $k\ne 0$.
We have just  seen that 
$H^0R\pi_*(\mathcal{K}_p)$ is locally free.  Thus since all $H^{i}R\pi_*(\mathcal{K}_p)$ are flat,
by \cite[Corollaire 6.9.9]{EGA3}, we have that cohomology and base change commute.
\end{proof}

\section{Local construction by multiplication table}\label{S:local}
Given a double ternary quadratic form $p\in \Sym^2 W \tensor U$ (with a given $\wedge^3 W \isom \wt U^*$), now that
we know that there is a natural quartic algebra $Q_p$ we could define the structure locally where $W$ and $U$ are free
by giving multiplication tables, as in the case of cubic algebras from binary cubic forms.

For a double ternary quadratic form over $\Z$ (and therefore with $W$ and $U$ necessarily free), Bhargava \cite[Equations (15) and (21)]{HCL3} gives
a ring structure on $\Z^4$ whose multiplication table is given in terms of the coefficients of $p$.  Each 
entry in the multiplication table is a polynomial in the coefficients of $p$.  This, of course, is the multiplication table we would impose for free $W$ and $U$ in the above local construction.  We will now see that this local construction agrees with the geometric construction we have given in Section~\ref{S:geom}.  We will show this by working over the universal algebra $R=\Z[\{a_{ij},b_{ij}\}_{1\leq i\leq j \leq 3}]$ for double ternary quadratic forms, and with the universal free form $u=\sum_{1\leq i\leq j \leq 3} a_{ij} x_ix_j y_1 + b_{ij} x_ix_j y_2$.

\begin{theorem}\label{SameOnUniversal}
 For the universal form $u$, the quartic algebra $Q_u$ is isomorphic to the quartic algebra over $R$ that is constructed above using Bhargava's multiplication tables.   
\end{theorem}
In particular, since our geometric construction of $Q_u$ is invariant under change of basis of $W$ and $U$ 
(respecting $\wedge^3 W \isom \wt U^*$), this gives a proof of the invariance of Bhargava's multiplication table under change of basis,
as long as the correct $\GL_3 \times \GL_2$ action is used.  Since all double ternary quadratic forms are locally
pull-backs from the universal form, and both the local construction by multiplication tables and the global geometric construction of Section~\ref{S:geom} respect base change, Theorem~\ref{SameOnUniversal} implies that the two constructions of quartic algebras from
double ternary quadratic forms agree.  We now prove Theorem~\ref{SameOnUniversal}. 

\begin{proof}
For the universal form $u$, the complex $\mathcal{K}_u$ used to define $Q_u$ is exact, and therefore $Q_u$ is just
the global functions on the scheme $S_u$ in $\P^2_R$ cut out by $A=\sum_{1\leq i \leq j \leq 3} a_{ij} x_ix_j $ and $B=\sum_{1\leq i \leq j \leq 3} b_{ij} x_ix_j$.
(We can just work in terms of global functions instead of the pushforward to the base since the base $\Spec R$ is affine.
Moreover, the multiplicative structure of the global functions of $S_u$ is the same as the induced multiplicative structure
on the hypercohomological construction of $Q_u$.)
We cover $S_u$ with open sets $\U_{x_i}$ coming from the usual open sets in $\P^2_R$.  As a first step, we will find 
$(f,g)\in \Gamma(\U_{x_i}) \times \Gamma(\U_{x_j})$ such that $f=g$ in $\Gamma(\U_{x_i} \cap\U_{x_j})$.
This will find all regular functions on $\U_{x_i} \cup\U_{x_j}$, and it will turn out that they all extend uniquely to
global functions on $S_u$.  Thus, we will have found all the regular functions on $S_u$. 
We will identify these regular functions with the  basis in Bhargava's quartic ring construction, and then it can be checked
that the multiplication tables agree. 

Let $i,j,k$ be some permutation of $1,2,3$.
We have that $$\Gamma(\U_{x_i})=R[x_j/x_i, x_k/x_i]/(A/x_i^2, B/x_i^2).$$
Let $I_i$ be the ideal $(A/x_i^2, B/x_i^2)$ of $R[x_j/x_i, x_k/x_i]$, and similarly for $I_j$.
Also, $$\Gamma(\U_{x_i} \cap\U_{x_j})=R[x_j/x_i, x_k/x_i, x_i/x_j]/(A/x_i^2, B/x_i^2).$$  If
we have $(f,g)\in \Gamma(\U_{x_i}) \times \Gamma(\U_{x_j})$ such that $f=g$ in $\Gamma(\U_{x_i} \cap\U_{x_j})$, then 
$f$ and $g$ are represented by polynomials $\tilde{f}\in R[x_j/x_i, x_k/x_i]$ and $\tilde{g}\in R[x_i/x_j, x_k/x_j]$
such the element $\tilde{f}-\tilde{g} \in R[x_j/x_i, x_k/x_i, x_i/x_j]$ is in the ideal $I=(A/x_i^2, B/x_i^2)$.
However, $\tilde{f}-\tilde{g}$ will not have any terms with an $x_i$ and an $x_j$ in the denominator.
We define $T_1$ to be the sub $R$-module of $I$ of elements that do not have any terms with
both an $x_i$ and an $x_j$ in the denominator.  The set $T_1$ gives all the relations between
polynomials representing elements in $\Gamma(\U_{x_i})$ and polynomials representing elements in $\Gamma(\U_{x_j})$.
We define $T_2$ to be the sub $R$-module of $T_1$ generated by the images of $I_i$ and $I_j$
under their natural inclusion into $R[x_j/x_i, x_k/x_i, x_i/x_j]$.  The set $T_2$ gives all the relations of 
$T_1$ that come from relations already in $\U_{x_i}$ and already in $\U_{x_j}$.  We now seek to determine $T_1/T_2$, 
which gives rise to all
pairs $(f,g)\in \Gamma(\U_{x_i}) \times \Gamma(\U_{x_j})$ such that $f=g$ in $\Gamma(\U_{x_i} \cap\U_{x_j})$
that are not functions on the base $\Spec R$.

We first define some notation to help us write down elements of $T_1/T_2$.  Let $A_{i^mj^n}=A\frac{x_k^{m+n-2}}{x_i^m x_j^n}$,
where the subscript $i^mj^n$ is a product of formal symbols, where a missing exponent denotes an exponent of 1.  
We define $B_{i^mj^n}$ analogously.
\begin{lemma}
 Let $t\in T_1/T_2$.  We can write 
$$t=
\mathop{\sum_{m,n\geq 1}}_{m+n\leq 3}
c_{m,n} A_{i^mj^n} + d_{m,n} B_{i^mj^n} \quad\quad \text{with } c_{m,n},d_{m,n}\in R.$$
\end{lemma}
\begin{proof}
Clearly we can write any $t$ in $I$ as such as sum over $m,n\in \Z$ with $m+n\geq 2$.  
Any term with $m\leq 0$ is in the image of $I_j$ and thus in $T_2$, and any term with 
$n\leq 0$ is in the image of $I_i$ and thus in $T_2$.  It remains to show that we do not need terms with
$m+n\geq 4$ in order to represent $t$.

We suppose for the sake of contradiction that a term with $m+n\geq 4$  was required, and we 
take a $t$ with $m+n$ maximal for this condition, and $m$ maximal given that.  
Then $c_{m,n} A_{i^mj^n}$ contributes a ${x_k^{m+n}}/{x_i^m x_j^n}$ term with coefficient
$c_{m,n} a_{kk}$ and $d_{m,n} B_{i^mj^n}$ contributes ${x_k^{m+n}}/{x_i^m x_j^n}$ term with coefficient
$d_{m,n} b_{kk}$.  No other terms of the summand for $t$ can contribute a term with $x_i^m x_j^n$ in the denominator, and so we must have
$c_{m,n}= r b_{kk}$ and $d_{m,n}= -r a_{kk}$ for some element $r\in R$.

Now we claim we did not need to use the terms $r b_{kk} A_{i^mj^n} -
r a_{kk}B_{i^mj^n}$ in the sum that represents $t$.  To prove this claim, we use the following identity
\begin{align*}
 &b_{kk}A_{i^mj^n}
-a_{kk}B_{i^m j^n}\\
=&-b_{ik}A_{i^{m-1}j^n}+a_{ik}B_{i^{m-1} j^n}-
b_{jk}A_{i^m j^{n-1}}+a_{jk}B_{i^m j^{n-1}}+a_{ij}B_{i^{m-1} j^{n-1}}
\\&-b_{ij}A_{i^{m-1} j^{n-1}}-b_{jj}A_{i^m j^{n-2}}
+a_{jj}B_{i^m j^{n-2}}-b_{i,i}A_{i^{m-2} j^n}+a_{ii}B_{i^{m-2} j^n}.
\end{align*}
This proves the lemma.
\end{proof}

The above lemma tells us that every element of $T_1/T_2$ can be written as an $R$-linear combination of 
$A_{ij}, B_{ij}, A_{i^2j}, B_{i^2j},
A_{ij^2},$ and $B_{ij^2}$.  Since only $A_{i^2j}$ and $B_{i^2j}$ have terms with $x_i^2 x_j$ in the denominator, 
we must have that $A_{i^2j}$ and $B_{i^2j}$ appear with coefficients $c_{2,1}$ and $d_{2,1}$ so as to cancel those terms out.
We can argue similarly for $A_{ij^2}$ and $B_{ij^2}$.
Thus, every element of $T_1/T_2$ can be written as a $R$ linear combination of $A_{ij}, B_{ij}, b_{kk}A_{i^2j}-a_{kk}B_{i^2j},$
 and $b_{kk}A_{ij^2}-a_{kk}B_{ij^2}$. We note that all four of $A_{ij}, B_{ij}, b_{kk}A_{i^2j}-a_{kk}B_{i^2j},$
 and $b_{kk}A_{ij^2}-a_{kk}B_{ij^2}$
have terms with a $x_ix_j$ denominator.  

We define some notation so we can write combinations of these elements down more easily.
For $i<j$, let $a_{ji}=a_{ij}$.
Let $\lambda^{\ell_1\ell_2}_{\ell_3\ell_4}=a_{\ell_1\ell_2}b_{\ell_3\ell_4}-
b_{\ell_1\ell_2}a_{\ell_3\ell_4}$.
We note that
\begin{align*}
 H_{i,j}&=b_{kk}A_{i^2j}-a_{kk}B_{i^2j}+b_{ik}A_{ij}-a_{ik} B_{ij}\\
&= \lambda^{jj}_{kk}x_jx_k/x_i^2
+ \lambda^{ij}_{kk}x_k/x_i
+ \lambda^{ii}_{kk}x_k/x_j
+\lambda^{jk}_{kk}x_k^2/x_i^2
+\lambda^{jj}_{ik}x_j/x_i
+ \lambda^{ij}_{ik}
+ \lambda^{ii}_{ik}x_i/x_j
+\lambda^{jk}_{ik}x_k/x_i\\
\text{           and }\\
H_{j,i}&=b_{kk}A_{ij^2}-a_{kk}B_{ij^2}+b_{jk}A_{ij}-a_{jk} B_{ij} 
\end{align*}
do not have any terms with both $x_i$ and $x_j$ in the denominator.  
Every element of $T_1/T_2$ can be written as a $R$ linear combination of $A_{ij}, B_{ij}, H_{i,j}$ and $H_{j,i}$, because
this is just a unipotent triangular transformation of the last list of four generators.
We have seen that $H_{i,j}$ and $H_{j,i}$ have no $x_k^2/x_ix_j$ terms, and $A_{ij}$ and $B_{ij}$ have $x_k^2/x_ix_j$ terms
with coefficients $a_{kk}$ and $b_{kk}$ respectively.
Since an element of $t$ does not have a term with $x_i x_j$ in the denominator, it can be written as a linear combination of
$H_{i,j}, H_{j,i}$ and $F_{ij}=F_{ji}=b_{kk}A_{ij}-a_{kk}B_{ij}$.  Moreover, $H_{i,j}, H_{j,i}$ and $F_{ij}$ are all in $T_1$.  
We now define $h_{\underline{i},j}$ to be the sum of terms in $H_{i,j}$ that do not have an $x_j$ in the denominator,
and $h_{i,\underline{j}}=H_{i,j}-h_{\underline{i},j}$.  We define
$f_{\underline{i}j}=f_{j\underline{i}}$ to be the sum of terms in $F_{ij}$ with $x_i$ in the denominator, so that
$f_{\underline{i}j}+f_{\underline{j}i}+\lambda^{ij}_{kk}=F_{ij}$.

We have now found that the pairs $(f,g)\in \Gamma(\U_{x_i}) \times \Gamma(\U_{x_j})$ such that $f=g$ in $\Gamma(\U_{x_i} \cap\U_{x_j})$ can be written in terms of four $R$-module generators: $$(1,1), (h_{\underline{i},j},-h_{i,\underline{j}}),
(h_{j,\underline{i}},-h_{\underline{j},i}),
(f_{\underline{i}j},-f_{\underline{j}i}+\lambda^{kk}_{ij}).$$
Letting $i$ and $j$ vary, this information is enough to determine the global functions on $S_u$.
In this case, it turns out that the regular functions on $\U_{x_i}$ that can be extended to
$\U_{x_j}$ are exactly the same as the regular functions on $\U_{x_i}$ that can be extended to
$\U_{x_k}$.  
In particular, in the polynomial ring $R[x_j/x_i, x_k/x_i]$, we can compute that 
$$
h_{\underline{i},j}+h_{\underline{i},k}=\lambda^{ii}_{jk}+a_{jk}B/x_i^2 -b_{jk}A/x_i^2
$$
and $$
h_{j,\underline{i}}=-f_{\underline{i}k}.
$$
Moreover, it will turn out that the extensions to $\U_{x_j}$ and $\U_{x_k}$
agree on their intersection.
We see that
the global functions of $S_u$ are generated as a $R$-module by four generators $g_1,g_2,g_3,g_4 
\in \Gamma(\U_{x_1}) \times \Gamma(\U_{x_2})\times \Gamma(\U_{x_3})$, whose components are given in the below table.

\begin{tabular}{l|l|l|l}
 & $\Gamma(\U_{x_1})$ & $\Gamma(\U_{x_2})$ & $\Gamma(\U_{x_3})$\\[2pt]
\hline
\hline
& & &\\[-8pt]
$g_1$ & 1 & 1 & 1 \\[4pt]
\hline
& & &\\[-8pt]
$g_2$ & $h_{\underline{1},2}=-h_{\underline{1},3}+\lambda^{11}_{23}$ & $-h_{1,\underline{2}}=f_{\underline{2}3}$ &
 $-f_{\underline{3}2}+\lambda^{11}_{23}=h_{1,\underline{3}}+\lambda^{11}_{23} $\\[4pt]
\hline
& & &\\[-8pt]
$g_3$ & $h_{2,\underline{1}}=-f_{\underline{1}3}$ & $-h_{\underline{2},1}=h_{\underline{2},3}+\lambda^{13}_{22}$ &
$-h_{2,\underline{3}}+\lambda^{13}_{22}= f_{\underline{3}1}+\lambda^{13}_{22}$ \\[4pt]
\hline
& & &\\[-8pt]
$g_4$ & $f_{\underline{1}2}=-h_{3,\underline{1}}$ & $-f_{\underline{2}1}+\lambda^{33}_{12}=
h_{3,\underline{2}}+\lambda^{33}_{12}$ & $-h_{\underline{3},2}+\lambda^{33}_{12}= h_{\underline{3},1}$ \\[4pt]
\end{tabular}

We now show that the $g_i$ are generators for a free $R$-module of rank 4.
Suppose for the sake of contradiction that there was a relation among these generators.
Then over the generic point of $R$ the global functions of $S_u$ would be a vector
space of at most dimension 3.  But we know from Section~\ref{SS:module} that the global
functions of $S_u$ are locally free four dimensional $R$ module, and thus
will be a four dimensional vector space over the generic point of $\Spec R$.

 To construct the multiplication table on
our four generators $g_i$ of the global functions on $S_u$, we 
can reduce to finding a multiplication table in the $\Gamma(\U_{x_1})$ component, since
the $g_i$ are $R$-linearly independent even in this component.
We can further reduce to finding the multiplication table over the generic point of $\Spec R$.
We first construct a multiplication table on $1,x_2/x_1, x_3/x_1, x_2x_3/x_1^2$ over the generic point of $\Spec R$.
To do this, we replace $A$ and $B$ by linear combinations of $A$ and $B$, one of which
has no $(x_2/x_1)^2$ term, and one of which has no $(x_3/x_1)^2$ term.  
Then on $\U_{x_1}$ over the generic point of $\Spec R$, we can write
all functions in terms of $1,x_2/x_1, x_3/x_1, x_2x_3/x_1^2$.  
We can then also write the $g_i$ in terms of $1,x_2/x_1, x_3/x_1, x_2x_3/x_1^2$,
and just apply this change of basis to the multiplication table to obtain a multiplication
table for the $g_i$.
If we take
$\alpha_1=-g_2$, $\alpha_2=-g_3$, and $\alpha_3=-g_4$, we obtain exactly the multiplication tables given by Bhargava in \cite[Equations (15) and (21)]{HCL3}. 
\end{proof}

In Section~\ref{SS:module}, we found that $Q_p/\OS$ is canonically isomorphic to $W^*$.  However, we also have explicit
basis for $Q_p/\OS$ when we have a basis for $W$.  We see how these bases are related.

\begin{theorem}\label{T:Cech}
For the universal form $u$, in the map $Q_p\ra W^*$ from Section~\ref{SS:module}, we have
\begin{align*}
 g_2 &\mapsto x_1^* \\
g_3 &\mapsto x_3^* \\
g_4 &\mapsto x_2^*. 
\end{align*}
\end{theorem}

\begin{proof}
 We compute the map in two steps.  We first find the map
$$R^0\pi_* (\mathcal{O}/u(\mathcal{O}(-2)^{\oplus 2})) \ra R^1\pi_* (\mathcal{O}(-2)^{\oplus 2}/u(\mathcal{O}(-4)))$$ and then the map
$$R^1\pi_* (\mathcal{O}(-2)^{\oplus 2}/u(\mathcal{O}(-4))) \ra R^2\pi_* (\mathcal{O}(-4)).$$  We compute each of the individual maps by using the snake
lemma on the Cech complex with the usual affine cover of $\P^2$.  
We summarize the computation in the charts below, which should be read from upper right to lower left.

\begin{tabular}{l|l|l|l|l|l}
 & $\mathcal{O}(-4)$ & $\mathcal{O}(-2)^{\oplus 2}$ & $\mathcal{O}(-2)^{\oplus 2}/u(\mathcal{O}(-4))$ & $\mathcal{O}$ & $\mathcal{O}/u(\mathcal{O}(-2)^{\oplus 2})$\\[2pt]
\hline
\hline
 $\Gamma(\U_{x_1})$ & & & & $h_{\underline{1},2}$ &  \\[4pt]
& & & & &\\[-8pt]
$ \times \Gamma(\U_{x_2})$ & & & & $-h_{1,\underline{2}}$ & $g_2$ \\[4pt]
& & & & &\\[-8pt]
 $\times \Gamma(\U_{x_3})$ & & & & $-f_{\underline{3},2}+\lambda^{11}_{23}$ &  \\[4pt]
& & & & &\\[-8pt]
\hline
 $\Gamma(\U_{x_1x_2})$ & & & $(\frac{a_{33}x_3}{x_1^2x_2}+\frac{a_{13}}{x_1x_2},$ & $h_{\underline{1},2}+h_{1,\underline{2}}=H_{1,2}$ &  \\[4pt]
& & & & &\\[-8pt]
 & & & $\frac{b_{33}x_3}{x_1^2x_2}+\frac{b_{13}}{x_1x_2})$ &  &  \\[4pt]
& & & & &\\[-8pt]
$ \times \Gamma(\U_{x_2x_3})$ & & &   & $-h_{1,\underline{2}}+f_{\underline{3},2}-\lambda^{11}_{23}$ &  \\[4pt]
& & & & &\\[-8pt]
 & & &  & $=f_{\underline{2},3}+f_{\underline{3},2}-\lambda^{11}_{23}$ &  \\[4pt]
& & & & &\\[-8pt]
 & & & $(\frac{a_{11}}{x_2x_3}, \frac{b_{11}}{x_2x_3})$ & $=F_{23}$ &  \\[4pt]
& & & & &\\[-8pt]
 $\times \Gamma(\U_{x_3x_1})$ & & & &$h_{\underline{1},2}+f_{\underline{3},2}-\lambda^{11}_{23}$ &  \\[4pt]
& & & & &\\[-8pt]
  & & & &$=-h_{\underline{1},3}-h_{1,\underline{3}}$ &  \\[4pt]
& & & & &\\[-8pt]
  & & & &$+a_{23}\frac{B}{x_1^2} -b_{23}\frac{A}{x_1^2}$ &  \\[4pt]
& & & & &\\[-8pt]
  & & & $-(\frac{a_{22}x_2}{x_1^2x_3}+\frac{a_{12}}{x_1x_3}, $ &$=-H_{1,3}$ &  \\[4pt]
& & & & &\\[-8pt]
  & & & $\frac{b_{22}x_2}{x_1^2x_3}+\frac{b_{12}}{x_1x_3})$ & &  \\[4pt]
& & & & &\\[-8pt]
  & & & $-(\frac{a_{23}}{x_1^2},\frac{b_{23}}{x_1^2})$ &$+a_{23}\frac{B}{x_1^2} -b_{23}\frac{A}{x_1^2}$ &  \\[4pt]
& & & & &\\[-8pt]
\hline
 $\Gamma(\U_{x_1x_2x_3})$ & $\frac{1}{x_1^2x_2x_3}$ & $\frac{A+B}{x_1^2x_2x_3}$ & & &  \\[4pt]
\end{tabular}

\begin{tabular}{l|l|l|l|l|l}
 & $\mathcal{O}(-4)$ & $\mathcal{O}(-2)^{\oplus 2}$ & $\mathcal{O}(-2)^{\oplus 2}/u(\mathcal{O}(-4))$ & $\mathcal{O}$ & $\mathcal{O}/u(\mathcal{O}(-2)^{\oplus 2})$\\[2pt]
\hline
\hline
 $\Gamma(\U_{x_1})$ & & & & $h_{2,\underline{1}}$ &  \\[4pt]
& & & & &\\[-8pt]
$ \times \Gamma(\U_{x_2})$ & & & & $-h_{\underline{2},1}$ & $g_3$ \\[4pt]
& & & & &\\[-8pt]
 $\times \Gamma(\U_{x_3})$ & & & & $f_{\underline{3},1}+\lambda^{13}_{22}$ &  \\[4pt]
& & & & &\\[-8pt]
\hline
 $\Gamma(\U_{x_1x_2})$ & & & $(\frac{a_{33}x_3}{x_1x_2^2}+\frac{a_{23}}{x_1x_2},$ & $h_{2,\underline{1}}+h_{\underline{2},1}=H_{2,1}$ &  \\[4pt]
& & & & &\\[-8pt]
 & & & $\frac{b_{33}x_3}{x_1x_2^2}+\frac{b_{23}}{x_1x_2})$ &  &  \\[4pt]
& & & & &\\[-8pt]
$ \times \Gamma(\U_{x_2x_3})$ & & &   & $-h_{\underline{2},1}-f_{\underline{3},1}-\lambda^{13}_{22}$ &  \\[4pt]
& & & & &\\[-8pt]
  & & & &$=h_{\underline{2},3}+h_{2,\underline{3}}$ &  \\[4pt]
& & & & &\\[-8pt]
  & & & &$-a_{13}\frac{B}{x_2^2} +b_{13}\frac{A}{x_2^2}$ &  \\[4pt]
& & & & &\\[-8pt]
  & & & $(\frac{a_{11}x_1}{x_3x_2^2}+\frac{a_{12}}{x_3x_2},$ &$=H_{2,3}$ &  \\[4pt]
& & & & &\\[-8pt]
  & & & $\frac{b_{11}x_1}{x_3x_2^2}+\frac{b_{12}}{x_3x_2})$ & &  \\[4pt]
& & & & &\\[-8pt]
  & & & $+(\frac{a_{13}}{x_2^2},\frac{b_{13}}{x_2^2})$ &$-a_{13}\frac{B}{x_2^2} +b_{13}\frac{A}{x_2^2}$ &  \\[4pt]
& & & & &\\[-8pt]

 $\times \Gamma(\U_{x_3x_1})$ & & & &$h_{2,\underline{1}}-f_{\underline{3},1}-\lambda^{13}_{22}$ &  \\[4pt]
& & & & &\\[-8pt]
 & & &  & $=-f_{\underline{1},3}-f_{\underline{3},1}-\lambda^{13}_{22}$ &  \\[4pt]
& & & & &\\[-8pt]
 & & & $-(\frac{a_{22}}{x_1x_3}, \frac{b_{22}}{x_1x_3})$ & $=-F_{13}$ &  \\[4pt]
& & & & &\\[-8pt]

\hline
 $\Gamma(\U_{x_1x_2x_3})$ & $\frac{1}{x_1x_2^2x_3}$ & $\frac{A+B}{x_1x_2^2x_3}$ & & &  \\[4pt]
\end{tabular}

\begin{tabular}{l|l|l|l|l|l}
 & $\mathcal{O}(-4)$ & $\mathcal{O}(-2)^{\oplus 2}$ & $\mathcal{O}(-2)^{\oplus 2}/u(\mathcal{O}(-4))$ & $\mathcal{O}$ & $\mathcal{O}/u(\mathcal{O}(-2)^{\oplus 2})$\\[2pt]
\hline
\hline
 $\Gamma(\U_{x_1})$ & & & & $-h_{3,\underline{1}}$ &  \\[4pt]
& & & & &\\[-8pt]
$ \times \Gamma(\U_{x_2})$ & & & & $-f_{\underline{2},1}+\lambda^{33}_{12}$ & $g_4$ \\[4pt]
& & & & &\\[-8pt]
 $\times \Gamma(\U_{x_3})$ & & & & $h_{\underline{3},1}$ &  \\[4pt]
& & & & &\\[-8pt]
\hline

 $\Gamma(\U_{x_1x_2})$& & &   & $-h_{3,\underline{1}}+f_{\underline{2},1}-\lambda^{33}_{12}$ &  \\[4pt]
& & & & &\\[-8pt]
 & & &  & $=f_{\underline{2},1}+f_{\underline{1},2}+\lambda^{12}_{33}$ &  \\[4pt]
& & & & &\\[-8pt]
 & & & $(\frac{a_{33}}{x_1x_2}, \frac{b_{33}}{x_1x_2})$ & $=F_{12}$ &  \\[4pt]
& & & & &\\[-8pt]

 $\times \Gamma(\U_{x_2x_3})$ & & & &$-h_{\underline{3},1}-f_{\underline{2},1}+\lambda^{33}_{12}$ &  \\[4pt]
& & & & &\\[-8pt]
  & & & &$=h_{\underline{3},2}+h_{2,\underline{3}}$ &  \\[4pt]
& & & & &\\[-8pt]
  & & & &$-a_{12}\frac{B}{x_3^2} +b_{12}\frac{A}{x_3^2}$ &  \\[4pt]
& & & & &\\[-8pt]
  & & & $(\frac{a_{11}x_1}{x_1x_3^2}+\frac{a_{13}}{x_2x_3}, $ &$=H_{3,2}$ &  \\[4pt]
& & & & &\\[-8pt]
  & & & $\frac{b_{11}x_1}{x_1x_3^2}+\frac{b_{13}}{x_2x_3})$ & &  \\[4pt]
& & & & &\\[-8pt]
  & & & $+(\frac{a_{12}}{x_3^2},\frac{b_{12}}{x_3^2})$ &$-a_{12}\frac{B}{x_3^2} +b_{12}\frac{A}{x_3^2}$ &  \\[4pt]
& & & & &\\[-8pt]

 $ \times \Gamma(\U_{x_3x_1})$ & & & $-(\frac{a_{22}x_2}{x_3^2x_1}+\frac{a_{23}}{x_1x_3},$ & $-h_{\underline{3},1}-h_{3,\underline{1}}=-H_{3,1}$ &  \\[4pt]
& & & & &\\[-8pt]
 & & & $\frac{b_{22}x_2}{x_3^2x_1}+\frac{b_{23}}{x_1x_3})$ &  &  \\[4pt]
& & & & &\\[-8pt]

\hline

 $\Gamma(\U_{x_1x_2x_3})$ & $\frac{1}{x_1x_2x_3^2}$ & $\frac{A+B}{x_1x_2x_3^2}$ & & &  \\[4pt]
\end{tabular}

\end{proof}

\section{Construction of the cubic resolvent}\label{S:constcubic}
In Section~\ref{S:cubic}, we have already given a geometric construction of a cubic algebra from a binary cubic form.  In Section~\ref{S:resolvent}, we defined the determinant
of a double ternary quadratic form $p$ to be a binary cubic form $\det(p)\in\Sym^3 U^* \tesnor (\wedge^2 U)$.
The cubic algebra $C$ of this binary cubic form can be constructed as described in Section~\ref{S:cubic}, and is the desired cubic resolvent.

We have $C/\OS \isom U$ (see \cite[Section 3.1]{binarynic} for a similar, but simpler argument to the one in Section~\ref{SS:module}).  Thus, $p$ gives the required
quadratic map from $Q/\OS$ to $C/\OS$.  The orientation isomorphism 
$\delta: \wedge^3 Q/\OS{\arisom} \wedge^2 C/\OS$ comes from the orientation on the 
double ternary quadratic form.  On any open set, we can check that
$\delta(x\wedge y \wedge xy)=p(x) \wedge p(y)$ by looking on a open subcover on
which $W$ and $U$ are trivial and pulling back from the universal form 
on each open set in that subcover.  It remains to check that 
$\delta(x\wedge y \wedge xy)=p(x) \wedge p(y)$ when $p$ is the universal ternary quadratic form, which can be checked explicitly given the multiplication table
of $Q_p$.  In particular, at the end of the proof of the Main Theorem in Section~\ref{S:Main}, we lay out a plan
to determine the multiplication table of $Q_p$ in terms of $p$.  The result agrees with the multiplication table given explicitly in \cite[Equations (15) and (21)]{HCL3}.
The expressions $\delta(x\wedge y \wedge xy)$ and $p(x) \wedge p(y)$ both represent linear maps from 
$\Sym_2(Q_p/\OS) \tesnor \Sym_2(Q_p/\OS)$ to $\wedge^4 Q_p$.  Thus it suffices to check that
these maps agree on a basis of global sections of $\Sym_2(Q_p/\OS) \tesnor \Sym_2(Q_p/\OS)$, since
in this case $Q_p/\OS$ is a free $\OS$ module.  This is easily checked, especially exploiting the 
symmetry of the situation.

\section{Main Theorem}\label{S:Main}
In this section, we prove the main theorem of this paper.
\begin{theorem}
There is an isomorphism between the moduli stack for quartic algebras with cubic resolvents and the moduli stack for double ternary quadratic forms.
In other words, for a scheme $S$ there is an equivalence between the category of quartic algebras with cubic resolvents and the category of
double ternary quadratic forms (with morphisms given by isomorphisms in both categories), and this natural equivalence commutes with base change in $S$.
\end{theorem}
\begin{proof}
Given a double ternary quadratic form $p$ over a base $S$, we have shown how
to construct a pair $(Q_p,C_p)$, and all aspects of the construction commute
with base change in $S$.  Given a pair $(Q,C)$ over $S$, we can just take the
quadratic map $\phi$ from $Q/\OS$ to $C/\OS$ to be our double ternary quadratic form
with $W=(Q/\OS)^*$ and $U=C/\OS$ 
(using the orientation $\wedge^3 Q/\OS\arisom \wt C/\OS$).  This construction clearly
commutes with base change.  

It remains to prove that the compositions of these two constructions
(in either order) are the identity.  To prove this, we rigidify the moduli 
problems.  A \emph{based double ternary quadratic form} is a ternary quadratic
form $p \in \Sym^2 W \tesnor U$ and a choice of bases $w_1,w_2,w_3$ and $u_1,u_2$
for $W$ and $U$ respectively as free $\OS$-modules, such that 
$(w_1\wedge w_2 \wedge w_3)\tensor (u_1\wedge u_2)$ corresponds to the identity
under the orientation isomorphism.  A \emph{based pair $(Q,C)$ of a
quadratic algebra and cubic resolvent} is a a pair $(Q,C)$ of quadratic algebra and cubic resolvent and choices of basis $q_1,q_2,q_3$ and $c_1,c_2$ for 
$Q/\OS$ and $C/\OS$ as free $\OS$-modules, such that
$(q_1\wedge q_2 \wedge q_3)$  corresponds to $(c_1\wedge c_2)$
under the orientation isomorphism.  We see that our constructions above
extend to the moduli stacks for these rigidified moduli problems.  
In particular, we obtain a basis for $Q/\OS$ as a dual basis for the basis of $W$
and vice versa.  

It now suffices to show that these constructions compose to the identity on the rigidified moduli stacks.  If we start with a double ternary quadratic form
$p\in \Sym^2 W \tensor U$,
we obtain a pair $(Q,C)$ whose quadratic map is given exactly by the form, and then
the construction of a form from $(Q,C)$ gives back exactly our original form.
The choices of bases for $W$ and $U$ and the orientation are clearly preserved under this composition.

We can start with a based pair $(Q,C)$, and then build another based pair
$(Q_\phi,C_\phi)$ from the quadratic map $\phi$ of $(Q,C)$, and we wish to
show that $(Q,C)$ and $(Q_\phi,C_\phi)$ are equal.  (We can use
the notion of equal instead of isomorphic since all of the objects are based.)
We have that $C$ and $C_\phi$ are both given as the cubic algebra corresponding
to $\Det(\phi)$ and thus are equal.  
The quadratic resolvent maps are the same, since $\phi$ carries through the two
constructions.
The orientation isomorphism are clearly the same
since they also carry through the constructions.  It remains to
show that the multiplication on $Q$ agrees with the multiplication on $Q_\phi$.
To do this, we will show that the condition $\delta(1\wedge x\wedge y \wedge xy)=\phi(x)\wedge \phi(y)$ determines the multiplication table on $Q$
from the resolvent map $\phi$.
Since $Q$ and $Q_\phi$ have the same resolvent map, this will show that
they are isomorphic as $\OS$-algebras.  

We let the quadratic map $\phi$ be written as $Ac_2+Bc_1$, where
$A=\sum_{1\leq i \leq j \leq 3} a_{ij} x_ix_j ,$ and $B=\sum_{1\leq i \leq j \leq 3} b_{ij} x_ix_j,$
and the $x_i$ are a dual basis for $q_i$ in $Q/\OS$.
We recall the notation $\lambda^{\ell_1\ell_2}_{\ell_3\ell_4}=a_{\ell_1\ell_2}b_{\ell_3\ell_4}-
b_{\ell_1\ell_2}a_{\ell_3\ell_4}$.  
We lift the basis $q_i$ of $Q/\OS$ to a basis of $Q$ uniquely so that
$q_1q_2$ has no $q_1$ or $q_2$ term and so that $q_1q_3$ has no $q_1$ term.
Let $m_{ij}^k$ be the coefficient of $q_k$ in the $q_i q_j$. 
From Equation (23) in \cite{HCL3}, we know that the constant coefficient of 
$q_i q_j$ in given as a polynomial in the various $m$ coefficients.
Thus, it remains to show that the $m_{ij}^k$ are determined by $\phi$.
We plug various $x$ and $y$ into $\delta(1\wedge x\wedge y \wedge xy)=\phi(x)\wedge \phi(y)$.
In the below, we always let $i,j,k$ be a permutation of $1,2,3$ and let $\pm$ be the sign of this permutation.
First, letting $x=q_i$ and $y=q_j$ gives $m_{ij}^k=\pm \lambda^{jj}_{ii}$.  
Then, letting $x=q_i+q_j$ and $y=q_i$ gives $m_{ii}^k=\pm \lambda^{ij}_{ii}$.
Next,  letting $x=q_i+q_k$ and $y=q_j$ gives $m_{jk}^k-m_{ij}^i=\pm \lambda^{jj}_{ik}$.
Using the choice of lift, which gives $m_{12}^1=m_{21}^2=m_{13}^1=0$, this determines all
$m_{ij}^i$.  Finally, letting $x=q_i+q_k$ and $y=q_i+q_j$ determines
$m_{ii}^i$ in terms of the $\lambda$'s and the $m$'s that we have already determined.
\end{proof}

\section{Appendix: Maps between locally free $\OS$-modules}\label{Appendix}
Let $S$ be a scheme.  In this appendix we will give several basic facts about 
maps between locally free $\OS$-modules.

\begin{lemma}\label{P:symline}
If $L$ is a locally free $\OS$-module and $V$ is a locally free rank $n$
$\OS$-module, then $\Sym^k(V\tensor L)\isom\Sym^k V \tensor L^{\tesnor k}$.
\end{lemma}
\begin{proof}
We have the canonical map
$$
\map{\Sym^k(V\tensor L)}{\Sym^k V \tensor \Sym^k L}{(v_1\tensor \ell_1)\cdots(v_k\tensor \ell_k)}{v_1\cdots v_k \tesnor \ell_1\cdots \ell_k},
$$
which we can check is an isomorphism on free modules and thus is an isomorphism
on locally free modules.
Moreover, we have that $L^{\tensor k} \isom \Sym^k L.$
We have the canonical quotient map
$
L^{\tensor k} \ra \Sym^k L
$
which is clearly an isomorphism for $L$ free of rank 1 and thus locally free of rank 1.
\end{proof}

\begin{lemma}\label{L:rk2}
If $V$ is a locally free $\OS$-module of rank two then $V \tensor \wedge^2 V^* \isom V^*$.
\end{lemma}
\begin{proof}
$$\map{V \tensor \wedge^2 V^*}{V^*}{v \tensor (\v_1\wedge \v_2) }{\v_1(v)\v_2-\v_2(v)\v_1}$$
We can define the canonical map
which is an isomorphism for free and thus locally free modules of rank 2. 
\end{proof}

We combine these two lemmas to obtain a corollary that is used throughout this paper.

\begin{corollary}\label{C:cubic}
If $V$ is a locally free $\OS$-module of rank two then $\Sym^3 V \tensor (\wedge^2 V)^{\tensor -2} \isom \Sym^3 V^* \tensor (\wedge^2 V^*)^{\tensor -1}$.
\end{corollary}

\begin{lemma}\label{dualsyms}
If $V$ is a locally free $\OS$-module, we have $(\Sym_n V)^* \cong \Sym^n V^*$.
\end{lemma}
\begin{proof}
We give a map from $\Sym^n V^*$ to $(\Sym_n V)^*$ as follows
$$
\v_1\v_2\cdots \v_n \mapsto (v_1\otimes \cdots \otimes v_n \mapsto \v_1(v_1)\v_2(v_2)\cdots \v_n(v_n) ).
$$
If we permute the $\v_i$ factors, we see the result does not change because the elements
of $\Sym_n V$ that we evaluate on are invariant with respect to this permutation.  When $V$ is free,
we can explicitly see that this map is an isomorphism.  
\end{proof}

\subsection{Degree $k$ maps}\label{SS:degkmaps}

Let $M$ and $N$ be locally free $\OS$-modules.
A linear map from $M$ to $R$ is equivalent to a global section of $M^*$.  In other words, sections of $M^*$ are the degree 1 functions on $M$.
We define the \emph{degree $n$} functions on $M$ as the global sections of $\Sym^n M^*$, symmetric polynomials in linear functions on $M$.

\begin{definition}
A \emph{degree $n$} map from $M$ to $N$ is a global section of $$\Sym^n M^* \tensor N \isom \mathcal{H}om(\Sym_n M, N).$$  
\end{definition}
Note that the identity map on $\Sym_n M$ gives a canonical degree $n$ map from $M$ to $\Sym_n M$.  

The language ``degree $n$ map from $M$ to $N$'' suggests that we should be able to evaluate such a thing on elements of $M$.

\begin{definition}
Given a degree $n$ map from $M$ to $N$ as an element $f\in\Hom(\Sym_n M, N)$,  the \emph{evaluation} of $f$
on an element of $M$ is $f(m\tensor \cdots \tensor m)$.
\end{definition}

When $M$ is free, say with generators $m_1,\dots,m_k$ and dual basis $\m_1,\dots\m_k$ of $M^*$, then we defined a degree $n$
map $f$ from $M$ to $R$ to be a homogeneous polynomial of degree $n$ in the $\m_1,\dots\m_k$.  If we evaluate
$f$ on $(c_1m_1 +\dots+ c_km_k)$ for arbitrary sections $c_i$ of $\OS$, we will have a degree $n$ polynomial in the $c_i$.
Replacing the $c_i$ in this polynomial by $\m_i$ we obtain the homogeneous polynomial of degree $n$ in the $\m_1,\dots\m_k$ 
which is the realization of $f$ as an element of $\Sym^n M^*$.

When $M$ is free, we may have a non-linear map $\rho :M\ra \OS$ (or $\rho:M\ra N$, but we take $N=\OS$ for simplicity)
and wish to realize it as the evaluation of a degree $n$ map.  We can consider $\rho(c_1m_1 +\dots+ c_km_k)$ for
arbitrary $c_i\in R$ and if $\rho(c_1m_1 +\dots+ c_km_k)$ is a degree $n$ polynomial in the $c_i$, we have
an
$f\in \Sym^n M^*$ (given by replacing the $c_i$ by $\m_i$) of which $\rho$ is the evaluation).

Since $M$ is locally free, we locally have $f\in \Sym^n M^*$ and see that
the above recipe is invariant under change of basis and so we have a global $f\in \Sym^n M^*$
(as long as everywhere locally where $M$ is free $\rho(c_1m_1 +\dots+ c_km_k)$ is a degree $n$ polynomial in the $c_i$).

As an example, we explicitly realize the determinant as a distinguished element of 
$$\Hom(\Sym_n \Hom(M,N), \Hom(\wn M,\wn N)).$$
Let $\phi_1\tensor\cdots\tensor \phi_n \in  \Hom(M,N)^{\tensor n}$.
Then we can map $\phi_1\tensor\cdots\tensor \phi_n$ to the element of $\Hom(\wn M,\wn N)$  which sends
$m_1 \wedge \cdots \wedge m_n$ to $ \phi_1(m_{1}) \wedge \cdots \wedge \phi_n(m_{n}).$
This will not be well-defined for $\phi_1\tensor\cdots\tensor \phi_n \in  \Hom(M,N)^{\tensor n}$, but
it will be well-defined when restricted to $\Sym_n \Hom(M,N)$.
\begin{align}\label{E:detdef}
\Sym_n \Hom(M,N)&& \lra&& \Hom(\wn M,\wn N)&\\
\phi_1\tensor\cdots\tensor \phi_n&& \mapsto&&  (m_1 \wedge \cdots \wedge m_n \mapsto &\phi_1(m_{1}) \wedge \cdots \wedge \phi_n(m_{n}) )
\notag
\end{align}
This is our realization of the determinant function (as opposed to the determinant of a specific
homomorphism)
as an element of $\Hom(\Sym_n \Hom(M,N), \Hom(\wn M,\wn N)).$
When we evaluate the determinant on a map $\phi \in\Hom(M,N)$, we have $\phi(m_{1}) \wedge \cdots \wedge \phi(m_{n}).$
For example, let $N$ and $M$ be free of rank 2.  Evaluating
our degree $2$ determinant map on a generic element of $\Hom(M,N)$ that sends $m_1$ to $an_1 +cn_2$ and $m_2$ to $bn_1+dn_2$,
we see that we obtain the element of $\Hom(\wedge^2 M, \wedge^2 N)$ that sends $m_1\wedge m_2$  to $(an_1 +cn_2)\wedge(bn_1+dn_2) =(ad-bc)n_1\wedge n_2$.

\subsection{Degree $k$ maps with coefficients}\label{SS:degkcoeff}

Recall that we have defined a degree $k$ map from a locally free $\OS$-module $M$
to a locally free $\OS$-module $V$ to be a linear map from
$\Sym_k M$ to $V$.  This is equivalent to a global section of $\Sym^k M^* \tensor V$.
We use the following proposition to show that we can ``add coefficients'' to
a degree $k$ map.

\begin{proposition}\label{P:coeffs}
 In the natural map
$$
\Sym_k(M\tesnor N) \ra M^{\tesnor k} \tesnor \Sym^k N,
$$
the image of $\Sym_k(M\tesnor N) $ is inside $\Sym_k M \tesnor \Sym^k N$.
\end{proposition}
\begin{proof}
We prove this proposition by checking the statement locally
where the modules are free.
If we symmetrize a pure tensor of basis elements in $(M\tesnor N)^{\tensor k}$, we see that
when we forget the terms from $N$ we still obtain an element of $\Sym_k M$.  Since
all of the terms in the symmetrization will have the same factor in $\Sym^k N$, this completes the proof.
\end{proof}

Thus, given a degree $k$ map from $M$ to $V$, we naturally obtain a degree $k$
map from $M\tesnor N$ to $V\tensor \Sym^k N$ (by composing
$\Sym_k (M\tesnor N) \ra \Sym_k M \tesnor \Sym^k N\ra V\tesnor \Sym^k N$).  
We call this construction \emph{using $V$ as coefficients}, because it is
as if we treat the elements of $V$ as formal ring elements.

\section{Appendix: Inherited algebra structure}\label{Appendix2}

Let $X$ be a scheme.  
An algebra structure on a chain complex $C$ of $\mathcal{O}_X$-modules is given by a map $C\tensor C \ra C$, which we call the multiplication.  (See \cite[2.7.1]{Weibel} for the definition of the tensor product of
two chain complexes.)  Associativity and commutativity are given by the commutativity of the expected diagrams built out of the multiplication map.  A unit for the algebra is given by a map $\mathcal{O}_X\ra C$ that satisfies the expected properties with respect to the multiplication.

If $\pi : X\ra Y$ is a morphism of schemes, such an algebra structure on $C$ is inherited by $R\pi_* C$ in the derived category of $Y$.  To be more precise, we let $Q$ be the localization functor that maps
complexes of $\mathcal{O}_X$-modules to the associated objects in the derived category of $X$.  
From the universal property of the derived tensor (see, e.g. \cite[10.5.1]{Weibel}), we have a morphism 
\begin{equation}\label{E:tensorQ}
 Q(C) \tensor Q(C) \ra Q(C\tensor C),
\end{equation}
where the $\tensor$ of the left denotes the total tensor in the derived category (see \cite[10.6]{Weibel} or \cite[II.4]{RD}).
From Equation~\eqref{E:tensorQ} composed with $Q(C\tensor C)\ra Q(C)$ from the multiplication map, we see that the algebra structure on $C$
is inherited by $Q(C)$ in the dervied category of $X$.

Next we see there is a map
\begin{equation}
 R\pi_* Q(C) \tensor R\pi_* Q(C) \ra R\pi_* (Q(C)\tensor Q(C))
\end{equation}
which can be obtained from the morphism $R\pi_* Q(C) \tensor R\pi_* Q(C) \ra R\pi_* (Q(C)\tensor L\pi^* R\pi_*Q(C))$ of the projection formula
(see \cite[10.8.1]{Weibel} \cite[II.5.6]{RD}) and the morphism $L\pi^* R\pi_*Q(C) \ra Q(C)$ that comes from the adjointness of $L\pi^*$ and $R\pi_*$ and the identity map
$R\pi_* \ra R\pi_*$ (c.f. \cite[10.7.1]{Weibel} \cite[II.5.10, II.5.11]{RD}).  Thus the algebra structure is inherited by $R\pi_* Q(C)$.  Finally, the natural map
\begin{equation}
 H^0( R\pi_* Q(C)) \tensor H^0(R\pi_* Q(C)) \ra H^0( R\pi_* Q(C) \tensor R\pi_* Q(C))
\end{equation}
shows how the algebra structure is inherited by $H^0( R\pi_* Q(C))$.  One can follow the diagrams to see that the associativity and commutativity of the multiplication map is inherited, along with the structure of the unit.

\section*{Acknowledgements}
The author would like to thank Manjul Bhargava for asking the questions that inspired this research, guidance along the way, and helpful feedback both on the ideas
and the exposition in this paper.  She would also like to thank Lenny Taelman for suggestions for improvements to the paper.  This work was done as part of the author's Ph.D. thesis at Princeton University, and during the work she was supported by an NSF Graduate Fellowship, an NDSEG Fellowship, an AAUW Dissertation Fellowship, and a Josephine De K\'{a}rm\'{a}n Fellowship.  This paper was prepared for submission while the author was supported by an American Institute of Mathematics Five-Year Fellowship.  The author would also like to thank the referee for making many suggestions that improved the paper.

\end{document}